\documentclass[11pt]{article}

\title{\textsc{Difference between families of weakly and strongly maximal integral lattice-free polytopes}}

\author{Gennadiy Averkov\footnote{Institute of Mathematical Optimization, Fakult\"at f\"ur Mathematik, Otto-von-Guericke-Universit\"at Magdeburg, Universit\"atsplatz~2,  39106 Magdeburg, Germany. Email: averkov@ovgu.de}}

\usepackage{amsthm,amsmath,amssymb,enumerate,graphicx,graphpap,empheq}
\usepackage{bbm} % for \mathbbm{1}
\usepackage{srcltx}

\usepackage{algorithm,algorithmic}
\floatname{algorithm}{Pseudocode}

\newcommand{\rmcmd}[1]{\mathop{\mathrm{#1}}\nolimits}
\newcommand{\conv}{\rmcmd{conv}}
\newcommand{\aff}{\rmcmd{aff}}

\newcommand{\relintr}{\rmcmd{relint}}

\newcommand{\R}{\mathbb{R}}
\newcommand{\N}{\mathbb{N}}
\newcommand{\Z}{\mathbb{Z}}

\newcommand{\Aff}{\rmcmd{Aff}}

\newcommand{\cardbig}[1]{\bigl|#1\bigr|}
\newcommand{\term}[1]{\emph{#1}}

\newcommand{\setcond}[2]{\left\{#1 \, : \, #2 \right\}}

\newcommand{\intr}{\rmcmd{int}}

\newcommand{\header}[1]{\textup{(#1)}}
\newcommand{\onevec}{\mathbbm{1}}
\newcommand{\cA}{\mathcal{A}}
\newcommand{\cM}{\mathcal{M}}

\newcommand{\cX}{\mathcal{X}}

\newcommand{\II}[1]{{#1}_I}
\newcommand{\cL}{\mathcal{L}}

\newtheorem{nn}{}%[section]
\newtheorem{theorem}[nn]{Theorem}

\newtheorem{proposition}[nn]{Proposition}
\newtheorem{lemma}[nn]{Lemma}
\newtheorem{remark}[nn]{Remark}
\newtheorem{question}[nn]{Question}
\newtheorem{corollary}[nn]{Corollary}

\theoremstyle{definition}

\newtheorem{definition}[nn]{Definition}
\newtheorem*{acknowledgements*}{Acknowledgements}

\numberwithin{equation}{section}

\begin{document}

\maketitle

\begin{abstract}
	A $d$-dimensional  closed convex set $K$ in $\R^d$ is said to be lattice-free if the interior of $K$ is disjoint with $\Z^d$. We consider the following two families of lattice-free polytopes: the family $\cL^d$ of integral lattice-free polytopes in $\R^d$ that are not properly contained in another integral lattice-free polytope and its subfamily $\cM^d$ consisting of integral lattice-free polytopes in $\R^d$ which are not properly contained in another lattice-free set. It is known that $\cM^d = \cL^d$ holds for $d \le 3$ and, for each $d \ge 4$, $\cM^d$ is a proper subfamily of $\cL^d$. We derive a super-exponential lower bound on the number of polytopes in $\cL^d \setminus \cM^d$ (with standard identification of integral polytopes up to affine unimodular transformations). 
\end{abstract}

\newtheoremstyle{itsemicolon}{}{}{\mdseries\rmfamily}{}{\itshape}{:}{ }{}

\newtheoremstyle{itdot}{}{}{\mdseries\rmfamily}{}{\itshape}{:}{ }{}

\theoremstyle{itdot}

%\newtheorem*{msc*}{2010 Mathematics Subject Classification} 

%\begin{msc*}
%	Primary 90C11;  Secondary 11H06, 11D45, \\ 52A40,  52B20,  52C07
%\end{msc*}

%\newtheorem*{keywords*}{Keywords}

%\begin{keywords*}
%	cutting-plane method; lattice diameter, lattice-free set; integral polytope; mixed-integer optimization
%\end{keywords*}

\section{Introduction}

By $|X|$ we denote the cardinality of a finite set $X$. Let $\N$ be the set of all positive integers and let $d \in \N$ be the dimension. We call elements of $\Z^d$ are called \term{integral points} or \emph{integral vectors}. We call a polyhedron $P \subseteq \R^d$ \term{integral} if $P$ is the convex hull of $P \cap \Z^d$. Let $\Aff(\Z^d)$ be the group of affine transformations $A: \R^d \to \R^d$ with $A(\Z^d) = \Z^d$. We call elements of $\Aff(\Z^d)$ \emph{affine unimodular transformations}. For a family $\cX$ of subsets of $\R^d$, we consider the family of equivalence classes
\[
\cX / \Aff(\Z^d) := \setcond{ \setcond{A (X)}{A \in \Aff(\Z^d)} }{X  \in \cX}
\]  
with respect to identification of the elements of $\cX$ up to affine unimodular transformations. A subset $K$ of $\R^d$ is called \term{lattice-free} if $K$ is closed, convex, $d$-dimensional and the interior of $K$ contains no points from $\Z^d$. A set $K$ is called \emph{maximal lattice-free} if $K$ is lattice-free and is not a proper subset of another lattice-free set. %Figure~\ref{ex:mlf:fig} provides several examples in dimension two.

Our objective is to study the relationship between the following two families of integral lattice-free polytopes:
\begin{itemize}
	\item The family $\cL^d$ of integral lattice-free polytopes $P$ in $\R^d$ such that there exists no integral \underline{lattice-free polytope} properly containing $P$. We call elements of $\cL^d$ \emph{weakly maximal} integral lattice-free polytopes.
	\item The family $\cM^d$ of integral lattice-free polytopes $P$ in $\R^d$ such that there exists no \underline{lattice-free set} properly containing $P$. We call the elements of $\cL^d$ \emph{strongly maximal} integral lattice-free polytopes. 
\end{itemize}

The family $\cL^d$ has applications in mixed-integer optimization, algebra and algebraic geometry; see \cite{MR3500323,AKW17}, \cite{blanco2016finiteness} and \cite{Treutlein10}, respectively. In \cite{AWW11,NZ11} it was shown that $\cL^d$ is finite up to affine unimodular transformations:
\begin{theorem} \header{\cite[Theorem~2.1]{AWW11}, \cite[Corollary~1.3]{NZ11}} \label{finiteness:thm} $\cL^d / \Aff(\Z^d)$ is finite.
\end{theorem}

Several groups of researchers are interested in enumeration of $\cL^d$, up to affine unimodular transformations, in fixed dimensions. This requires understanding geometric properties of $\cL^d$. Currently, no explicit description of $\cL^d$ is available for dimensions $d \ge 4$ and, moreover, it is even extremely hard to decide if a given polytope belongs to $\cL^d$. A brute-force algorithm based on volume bounds for $\cL^d$ (provided in \cite{NZ11}) would have doubly exponential running time in $d$. In contrast to $\cL^d$, its subfamily $\cM^d$ is easier to deal with. Lov\'asz's characterization \cite[Proposition~3.3]{MR1114315} of maximal lattice-free sets leads to a straightforward geometric description of polytopes belong to $\cM^d$. This characterization can be used to decide whether a given polytope is an element of $\cM^d$ in only exponential time in $d$. Thus, while enumeration of $\cM^d$ in fixed dimensions is a hard task, too, enumeration of $\cL^d$ is even more challenging. 

For a given dimension $d$, it is a priori not clear whether or not $\cM^d$ is a proper subset of $\cL^d$. Recently, it has been shown that the inequality $\cM^d = \cL^d$ holds if and only if $d \le 3$. The equality $\cM^d = \cL^d$ is rather obvious for $d  \in \{1,2\}$, as it is not hard to enumerate $\cL^d$ in these very small dimensions and to check that every element of $\cL^d$ belongs to $\cM^d$. Starting from dimension three, the problem gets very difficult. Results in \cite{AWW11} and \cite{AKW17} establish the equality $\cM^3 = \cL^3$ and enumerate $\cL^3$, up to affine unimodular transformations. As a complement, in \cite[Theorem~1.4]{NZ11} it was shown that for all $d \ge 4$ there exists a polytope belonging to $\cL^d$ but not to $\cM^d$. 

While Theorem~1.4 in \cite{NZ11} shows that $\cL^d$ and $\cM^d$ are two different families, it does not provide information on the number of polytopes in $\cL^d$ that do not belong to $\cM^d$. Relying on a result of Konyagin \cite{Konyagin14}, we will show that, asymptotically, the gap between $\cL^d$ and $\cM^d$ is very large. 

For $a_1,\ldots,a_d >0$, we introduce
\begin{equation} \label{kappa:def}
\kappa(a):=\kappa(a_1,\ldots,a_d) = \frac{1}{a_1} + \cdots + \frac{1}{a_d}.
\end{equation}
Reciprocals of positive integers are sometimes called Egyptian fractions. Thus, if $a \in \N^d$, then $\kappa(a)$ is a sum of $d$ Egyptian fractions. 
We consider the set 
\begin{equation} \label{cA:def}
\cA_d := \setcond{(a_1,\ldots,a_d) \in \N^d}{a_1 \le \ldots \le a_d, \ \kappa(a_1,\ldots,a_d)=1}
\end{equation}
of all different solutions of the Diophantine equation 
\[
\kappa(x_1,\ldots,x_d)=1
\] in the unknowns $x_1,\ldots,x_d \in \N$. The set $\cA_d$ represents possible ways to write $1$ as a sum of $d$ Egyptian fractions. It is known that $\cA_d$ is finite.  Our main result allows is a lower bound on the cardinality of $(\cL^d\setminus \cM^d) / \Aff(\Z^d)$:

\begin{theorem} \label{main}
	\(
	\cardbig{(\cL^{d+5}\setminus \cM^{d+5}) / \Aff(\Z^{d+5})} \ge \cardbig{\cA_d}.
	\)
\end{theorem}

The proof of Theorem~\ref{main} is constructive. This means that, for every $a \in \cA_d$, we generate an element in $P_a \in \cL^{d+5} \setminus \cM^{d+5}$ such that for two different elements $a$ and $b$ of $\cA_d$, the respective polytopes $P_a$ and $P_b$ do not coincide up to affine unimodular transformations. The proof of Theorem~\ref{main} is inspired by the construction in \cite{NZ11}. Using lower bounds on $\cardbig{\cA_d}$ from \cite{Konyagin14}, we obtain the following asymptotic estimate:

\begin{corollary} \label{lower:bound:for:difference}
	\(
	\ln \ln \cardbig{\bigl( \cL^d \setminus \cM^d \bigr) / \Aff(\Z^d)} = \Omega \left( \frac{d}{\ln d} \right) \),
	as $d \to \infty$. 
\end{corollary}

\paragraph*{Notation.}
We view the elements of $\R^d$ as columns. By $o$ we denote the zero vector and by $e_1,\ldots,e_d$ the standard basis of $\R^d$. If $x \in \R^d$ and $i \in \{1,\dots,d\}$, then $x_i$ denotes the $i$-th component of $x$. The relation $a \le b$ for $a, b \in\R^d$ means $a_i \le b_i$ for every $i \in \{1,\ldots,d\}$. The relations $\ge, >$ and $<$ on $\R^d$ are introduced analogously. The abbreviations $\aff, \conv, \intr$ and $\relintr$ stand for the affine hull, convex hull, interior and relative interior, respectively. 

%Let $p_1,\ldots,p_m \in \R^d$ with $m \in \N$ and let $x$ be a \term{convex combination} of $p_1,\ldots,p_m$, that is, there exist coefficients $\alpha_1,\ldots,\alpha_m \ge 0$ with
%\begin{align}
%	1 & = \alpha_1 + \cdots + \alpha_m, \label{sum:is:one} \\
%	x & = \alpha_1 p_1 + \cdots + \alpha_m p_m. \label{x:lin:comb}
%\end{align}
%If for $x$ as above all $\alpha_1,\ldots,\alpha_m$ are strictly positive we call $x$ a \term{positive convex combination} of $p_1,\ldots,p_m$. For affinely independent $p_1,\ldots,p_m$, the set $T:= \conv \{p_1,\ldots,p_m\}$ is a simplex; in this case $\alpha_1,\ldots,\alpha_m$ are uniquely determined by \eqref{sum:is:one} and \eqref{x:lin:comb} and we call these coefficients the \term{barycentric coordinates} of $x$ with respect to the simplex $T$.

\section{An approach to construction of polytopes in $\cL^d \setminus \cM^d$} 

We will present a systematic approach to construction of polytopes in $\cL^d \setminus \cM^d$, but first we discuss general maximal lattice-free sets. 

\begin{definition}
	Let $P$ be a lattice-free polyhedron in $\R^d$. We say that a facet $F$ of $P$ is \emph{blocked} if the relative interior of $F$ contains an integral point.
\end{definition}

Maximal lattice-free sets can be characterized as follows:
\begin{proposition} \header{\cite[Proposition~3.3]{MR1114315}.} 
	\label{lovasz}
	Let $K$ be a $d$-dimensional closed convex subset of $\R^d$. Then the following conditions are equivalent.
	\begin{enumerate}[(i)]
		\item $K$ is maximal lattice-free, 
		\item $K$ is a lattice-free polyhedron such that every facet of $K$ is blocked.
	\end{enumerate}
\end{proposition}

It can happen that some facets of a maximal lattice-free polyhedron are more than just blocked. We introduce a respective notion. Recall that the \term{integer hull} $\II{K}$ of a compact convex set $K$ in $\R^d$ is defined by 
\[
\II{K}:= \conv( K \cap \Z^d).
\]

\begin{definition}
	Let $P$ be a $d$-dimensional lattice-free polyhedron in $\R^d$. A facet $F$ of $P$ is called \term{strongly blocked} if $\II{F}$ is $(d-1)$-dimensional and $\Z^d \cap \relintr \II{F} \ne \emptyset$. The polyhedron $P$ is called \term{strongly blocked} if all facets of $P$ are strongly blocked.
\end{definition}

The following proposition extracts the geometric principle behind the construction from \cite[Section~3]{NZ11}. (Note that arguments in \cite[Section~3]{NZ11} use an algebraic language.) 

\begin{proposition} \label{sufficient:to:be:in:M:Z}
	Let $P$ be a strongly blocked lattice-free polytope in $\R^d$. Then $\II{P} \in \cL^d$. Furthermore, if $\II{P}$ is not integral, then $\II{P} \not\in \cM^d$. 
\end{proposition}
\begin{proof}
	In order to show $\II{P} \in \cL^d$ it suffices to verify that, for 
	every $z \in \Z^d$ such that $\conv (\II{P} \cup \{z\})$ is lattice-free, one necessarily has $z \in \II{P}$. If $z \not\in \II{P}$, then $z \not\in P$ and so, for some facet $F$ of $P$, the point $z$ and the polytope $P$ lie on different sides of the hyperplane $\aff F$. Then $\emptyset \ne \Z^d \cap \relintr \II{F} \subseteq \intr( \conv (P \cup \{z\}))$, yielding a contradiction to the choice of $z$. Thus, for every facet $F$ of $P$, $z$ and $P$ lie on the same side of $\aff F$. It follows $z \in P$. Hence $z \in P \cap \Z^d \subseteq \II{P}$. 
	
	If $P$ is not integral, then $\II{P} \not\in \cM^d$ since $\II{P} \varsubsetneq P$ and $P$ is lattice-free. 
\end{proof}

\section{Lattice-free axis-aligned simplices}

For $a \in \R_{>0}^d$, the $d$-dimensional simplex
\[
T(a):= \conv \{ o, a_1 e_1, \ldots, a_d e_d \}.
\]
is called \term{axis-aligned}. The proof of the following proposition is straightforward.
\begin{proposition} \label{coordinate:simplices}
	For $a \in \R_{>0}^d$, the following statements hold. 
	\begin{enumerate}[I.]
		\item The simplex $T(a)$ is a  lattice-free set if and only if $\kappa(a) \ge 1$.
		\item The simplex $T(a)$ is a maximal lattice-free set if and only if $\kappa(a)=1$.
	\end{enumerate}
\end{proposition}

We introduce transformations which preserve the values of $\kappa$. The transformations arise from the following trivial identities for $t>0$:
\begin{align}
	\frac{1}{t} & = \frac{1}{t+1} + \frac{1}{t(t+1)}, \label{triv:1} \\
	\frac{1}{t} & = \frac{1}{t+2} + \frac{1}{t(t+2)} + \frac{1}{t(t+2)}, \label{triv:2} \\
	\frac{1}{t} & = \frac{2}{3 t} + \frac{1}{3 t}. \label{triv:3}
\end{align}
Consider a vector $a \in \R_{>0}^d$. By \eqref{triv:1}, if $t$ is a component of $a$, we can replace this component with two new components $t+1$ and $t(t+1)$ to generate a vector $b \in \R_{>0}^{d+1}$ satisfying $\kappa(b)=\kappa(a)$. Identities \eqref{triv:2} and \eqref{triv:3} can be applied in a similar fashion. For every $d \in \N$, with the help of \eqref{triv:1}--\eqref{triv:3}, we introduce the following maps:
\begin{align}
	\phi_d &:\R_{>0}^d \rightarrow \R_{>0}^{d+1}, & 
	\phi_d(a) &:= \begin{pmatrix} a_1 \\ \vdots \\ a_{d-1} \\ a_d+1 \\ a_d(a_d+1) \end{pmatrix}, \label{phi:def} 
	\\
	\psi_d &: \R_{>0}^d \rightarrow \R_{>0}^{d+3}, & \psi_d(a)  &: = \begin{pmatrix} a_1 \\ \vdots \\ a_{d-1} \\ a_d+3 \\ a_d(a_d+1)
	\\ (a_d+1)(a_d+3) \\ (a_d+1)(a_d+3) \end{pmatrix} , \label{psi:def} \\
	\xi_d &: \R_{>0}^d \rightarrow \R_{>0}^{d+1} & 
	\xi_d(a)  &:= \begin{pmatrix} a_1 \\ \vdots \\ a_{d-1} \\ \frac{3}{2} a_d \\ 3a_d \end{pmatrix}. \label{xi:def}
\end{align}
%\begin{align}
%	\phi_d &:\R_{>0}^d \rightarrow \R_{>0}^{d+1}, & \phi_d(a) := \Bigl(&a_1,\ldots,a_{d-1}, a_d+1,a_d(a_d+1)\Bigr)^\top, \label{phi:def} \\
%	\psi_d &: \R^d \rightarrow \R^{d+3}, & \psi_d(a)  : =\Bigl(&a_1,\ldots,a_{d-1},a_d+3,a_d(a_d+1), \nonumber \\
%	& & &(a_d+1)(a_d+3),(a_d+1)(a_d+3)\Bigr)^\top, \label{psi:def} \\
%	\xi_d &: \R^d \rightarrow \R^{d+1} & \xi_d(a)  := \Bigl(&a_1,\ldots,a_{d-1}, \frac{3}{2} a_d, 3a_d\Bigr)^\top. \label{xi:def}
%\end{align}
The map $\phi_d$ replaces the component $a_d$ by two other components based on \eqref{triv:1}, while $\xi_d$ replaces $a_d$ based on \eqref{triv:3}. The map $\psi_d$ acts by replacing  the component $a_d$ based on \eqref{triv:1} and then replacing the component $a_d+1$ based on \eqref{triv:2}. Identities \eqref{triv:1}--\eqref{triv:3} imply 
\begin{equation} \label{kappa:preserving}
	\kappa(\phi_d(a)) = \kappa(\psi_d(a))=\kappa(\xi_d(a)))= \kappa(a).
\end{equation}

\begin{lemma} 
	\label{replacing:components}
	Let $P=T(\xi_d(a))$, where $a \in \cA_d$ and $d \ge 2$. Then
	$P$ is a strongly blocked lattice-free $(d+1)$-dimensional polytope. Furthermore, if $a_d$ is odd, $P$ is not integral. 
\end{lemma}
\begin{proof} 
	In this proof, we use the \term{all-ones vector} 
	\[
		\onevec_d := \begin{pmatrix} 1 \\ \vdots \\ 1 \end{pmatrix} \in \R^d.
	\]
	
	For the sake of brevity we introduce the notation $t:=a_d$. One has $1= \kappa(a) = \sum_{i=1}^d \frac{1}{a_i} \ge \sum_{i=1}^d \frac{1}{t} = \frac{d}{t}$, which implies $t \ge d \ge 2$.  By \eqref{kappa:preserving}, one has $\kappa(\xi_d(a))=1$ and so, by Proposition~\ref{coordinate:simplices}, $P$ is maximal lattice-free. 
	
	If $t$ is even, the polytope $P$ is integral and hence every facet of $P$. In view of Proposition~\ref{lovasz}, integral maximal lattice-free polytopes are strongly blocked, and so we conclude that $P$ is strongly blocked. 
	
	Assume that $t$ is odd, then the polytope $P$ has one non-integral vertex. In this case, we need to look at facets of $P$ more closely, to verify that $P$ is strongly blocked. We consider all facets of $P$.
	\begin{enumerate}
		\item The facet \(
		F=\conv \bigl\{o,a_1 e_1,\ldots,a_{d-1} e_{d-1}, 3 t e_{d+1}\bigr\}
		\)
		is a $d$-dimensional integral integral axis-aligned simplex. Since
		\[
		\kappa(a_1,\ldots,a_{d-1},3t) < 1,
		\]
		the integral point $e_1 + \cdots + e_{d-1} + e_{d+1}$ is in the relative interior of $F$. Hence, $F$ is strongly blocked. 
		\item 	The facet 
		\(
			F = \conv \Bigl\{o,a_1 e_1,\ldots,a_{d-1} e_{d-1}, \frac{3}{2} t e_d\Bigr\}
		\)
		contains the $d$-dimensional integral axis-aligned simplex
		\[
			G:=\conv \Bigl\{o, a_1 e_1,\ldots,a_{d-1} e_{d-1}, \frac{3t-1}{2} e_d\Bigr\},
		\]
		as a subset. In view of $t \ge 2$, we have 
		\[
		\kappa\Bigl(a_1,\ldots,a_{d-1},\frac{3t-1}{2}\Bigr) < 1,
		\]
		which implies that the integral point $e_1 + \dots + e_d$ is in the relative interior of $G$. It follows that $F$ is strongly blocked. 
		\item 	The facet
		\(
			F:= \conv \Bigl\{a_1e_1,\ldots,a_{d-1} e_{d-1}, \frac{3}{2} t e_d, 3 t e_{d+1} \Bigr\}
		\)
		contains the integral $d$-dimensional simplex 
		\[
			G: = \conv \Bigl\{a_1 e_1,\ldots,a_{d-1} e_{d-1}, \frac{3 t-1}{2} e_d + e_{d+1}, 3 t e_{d+1}\Bigr\}.
		\]
		as a subset. It turns out that $\onevec_{d+1}$ is the relative interior of $G$, because $\onevec_{d+1}$ is a convex combination of the vertices of $\relintr G$, with positive coefficients. Indeed, the equality
		\[
			\onevec_{d+1} = \sum_{i=1}^{d-1} \frac{1}{a_i} \bigl(a_i e_i\bigr) + \lambda \Bigl(\frac{3t-1}{2} e_d + e_{d+1}\Bigr) + \mu \bigl(3 t e_{d+1}\bigr)
		\]
		holds for $\lambda= \frac{2}{3t-1}$ and $\mu= \frac{t-1}{t(3t-1)}$, where 
		\[
			\sum_{i=1}^{d-1} \frac{1}{a_i} + \lambda + \mu = 1.
		\]
		\item It remains to consider faces $F$ with the vertex set $\left\{ o, a_1 e_1,\ldots,a_d e_d, \frac{3}{2} t e_d, 3 t e_{d+1} \right\} \setminus \{a_i e_i\}$, where $i \in \{1,\ldots,d+1\}$. Without loss of generality, let $i=1$ so that $F = \conv \left\{ o, a_2 e_2,\ldots, \frac{3}{2} t e_d, 3 t e_{d+1} \right\}$. This facet contains the integral $d$-dimensional simplex 
		\[
			G: = \conv \Bigl\{o, a_2 a_2,\ldots,a_{d-1} e_{d-1}, \frac{3 t-1}{2} e_d + e_{d+1}, 3 t e_{d+1}\Bigr\}.
		\]
		Similarly to the previous case, one can check that $e_2 + \cdots + e_{d+1}$ is an integral point in the relative interior of $G$. Consequently, $F$ is strongly blocked. 
	\end{enumerate}
\end{proof}

\section{Proof of the main result}

For $d \ge 4$, Nill and Ziegler \cite{MR1138580} construct one vector $a\in \R_{>0}^d$ with $\II{T(a)} \in \cL^d \setminus \cM^d$.  We generalize this construction and provide many further vectors $a$ with the above properties. We will also need to verify that for difference choices of $a$, we get essentially different polytopes $\II{T(a)}$.

\begin{lemma}
	\label{int:hull:strongly:blocked}
	Let $P$ and $Q$ be $d$-dimensional strongly blocked lattice-free polytopes such that for their integral hulls the equality $\II{Q}=A(\II{P})$ holds for some $A \in \Aff(\Z^d)$. Then $Q=A(P)$. 
\end{lemma}
\begin{proof}
	Since $A$ is an affine transformation, we have \[A(P_I) = A(\conv (P \cap \Z^d)) = \conv A(P \cap \Z^d).\]
	Using $A \in \Aff(\Z^d)$, it is straightforward to check the equality $A(P \cap \Z^d) = A(P) \cap \Z^d$. We thus conclude that $A(\II{P}) = \II{A(P)}$. The assumption $\II{Q} = A(\II{P})$ yields $\II{Q} = \II{A(P)}$. Since $P$ is strongly blocked lattice-free, $A(P)$ too is strongly blocked lattice-free. We thus have the equality $\II{Q}=\II{A(P)}$ for  strongly blocked lattice-free polytopes $Q$ and $A(P)$. To verify the assertion, it suffices to show that a strongly blocked lattice-free polytope $Q$ is uniquely determined by the knowledge of its integer hull $\II{Q}$. This is quite easy to see. For every strongly blocked facet $G$ of $\II{Q}$, the affine hull of $G$ contains a facet of $Q$. Conversely, if $F$ is an arbitrary facet of $Q$, then $G= \II{F}$ is a strongly blocked facet of $\II{Q}$. Thus, the knowledge of $\II{Q}$ allows to determine affine hulls of all facets of $Q$. In other words, $\II{Q}$ uniquely determines a hyperplane description of $Q$. 
\end{proof}

\begin{lemma}
	\label{equiv:axis-aligned}
	Let $a,b \in \R_{>0}^d$ be such that the equality $T(b)= A(T(a))$ holds for some $A \in \Aff(\Z^d)$. Then $a$ and $b$ coincide up to permutation of components. 
\end{lemma}
\begin{proof}
	We use induction on $d$. For $d=1$, the assertion is trivial. Let $d \ge 2$. One of the $d$ facets of $T(a)$ containing $o$ is mapped by $A$ to a facet of $T(b)$ that contains $o$. Without loss of generality we can assume that the facet $T(a_1,\ldots,a_{d-1}) \times \{0\}$ of $T(a)$ is  mapped to the facet $T(b_1,\ldots,b_{d-1}) \times \{0\}$ of $T(b)$. By the inductive assumption, $(a_1,\ldots,a_{d-1})$ and $(b_1,\ldots,b_{d-1})$ coincide up to permutation of components. Since unimodular transformations preserve the volume, $T(a)$ and $T(b)$ have the same volume. This means, $\prod_{i=1}^d a_i = \prod_{i=1}^d b_i$. Consequently, $a_d=b_d$ and we conclude that $a$ and $b$ coincide up to permutation of components. 
\end{proof}

%Lemma~\ref{injectivity:lemma}.V allows to find explicit bounds for $\cardbig{(\cL^d \setminus \cM^d) / \Aff(\Z^d)}$ for concrete values of $d$. For example, it is known that $\card{\cA^6} = 3462$ (see \cite[p.163]{MR1299330}). Thus, for $d=11$ one has $\cardbig{(\cL^{d}\setminus \cM^{d})/ \Aff(\Z^{d})} \ge 3462$.

\begin{proof}[Proof of Theorem~\ref{main}]
	For every $a \in \cA_d$, we introduce the $(d+5)$-dimensional integral lattice-free polytope
	\[
		P_a:=\II{T(\eta(a))},
	\]
	where
	\[
		\eta(x) := \xi_{d+4}(\psi_{d+1}(\phi_d(x)))
	\]
	and the functions $\xi_{d+4}, \psi_{d+1}$ and $\phi_d$ are defined by \eqref{phi:def}--\eqref{xi:def}.
	
	By \eqref{kappa:preserving} for each $a \in \cA_d$, we have $\kappa(\eta(a))=1$. For $a \in \cA_d$ the last component of $\phi_d(a)$ is even. This implies that the last component of $\psi_{d+1} (\phi_d(a))$ is odd. Thus, by Lemma~\ref{replacing:components}, $T(\eta(a))$ is strongly blocked lattice-free polytope which is not integral.
	
	Let $a, b \in \cA_d$ be such that the polytopes $P_a$ and $P_b$ coincide up to affine unimodular transformations. Then, by Lemma~\ref{int:hull:strongly:blocked}, $T(\eta(a))$ and $T(\eta(b))$ coincide up to affine unimodular transformations. But then, by Lemma~\ref{equiv:axis-aligned}, $\eta(a)$ and $\eta(b)$ coincide up to permutations. Since the components of $a$ and $b$ are sorted in the ascending order, the components of $\eta(a)$ and $\beta(b)$ too are sorted in the ascending order. Thus, we arrive at the equality $\eta(a) = \eta(b)$, which implies $a=b$. 
	
	In view of Proposition~\ref{sufficient:to:be:in:M:Z}, each $P_a$ with $a \in \cA_d$ belongs to $\cL^d$ but not to $\cM^d$. Thus, the equivalence classes of the polytopes $P_a$ with $a \in \cA_d$ with respect to identification up to affine unimodular transformations form a subset of 
	$(\cL^{d+5} \setminus \cM^{d+5} ) / \Aff(\Z^{d+5})$ of cardinality $|\cA_d|$. This yields the desired assertion. 
\end{proof}

\begin{proof}[Proof of Corollary~\ref{lower:bound:for:difference}]
	The assertion is a direct consequence of Theorem~\ref{main} and the asymptotic estimate 
	\[
		\ln \ln |\cA_d|  = \Omega \left( \frac{d}{\ln d} \right)
	\]
	of Konyagin \cite[Theorem~1]{Konyagin14}.
\end{proof}

\begin{remark}
	In view of the upper bound \( \ln \ln |\cA_d| = O(d)\) by S\'andor \cite[Theorem~2]{MR2025624},
	 the lower bound of Konyagin is optimal up to the logarithmic factor in the denominator.
\end{remark}

Since all known elements of $\cL^d$ are of the form $\II{P}$, for some strongly blocked lattice-free polytope $P$, we ask the following
\begin{question}
	Do there exist polytopes $L \in \cL^d$ which cannot be represented as $L=\II{P}$ for any strongly blocked lattice-free polytope $P$?  
\end{question}

If there is a gap between the families $\cL^d$ and the family 
\[
	\setcond{\II{P}}{P \subseteq \R^d \ \text{strongly blocked lattice-free polytope}},
\]
then it would be interesting to understand how irregular the polytopes from this gap can be. For example, one can ask the following 

\begin{question}
	Do there exist polytopes $L \in \cL^d$ with the property that no facet of $L$ is blocked?
\end{question}

\subsection*{Acknowledgements} 

%\begin{acknowledgements*}
	I would like to thank Christian Wagner for valuable comments.
%\end{acknowledgements*}

\bibliographystyle{amsalpha}
\bibliography{literature}

\end{document}